\documentclass[12pt,twoside,a4paper]{article}

\usepackage{a4,enumerate}
\usepackage[noadjust]{cite}
\usepackage{amssymb,amsmath,amsthm}
\usepackage{comment}
\usepackage{calc}
\usepackage{xcolor}

  \newcommand\mytitle{The essential norm of multiplication operators on $L_p(\mu)$} 
\newcommand\lhead{J. Voigt}
\newcommand\rhead{Essential norm of multiplication operators}
\swapnumbers
\numberwithin{equation}{section}

\newtheorem{theorem}{Theorem}[section]

\newtheorem{lemma}[theorem]{Lemma}
\theoremstyle{definition}

\newtheorem{remark}[theorem]{Remark}

 \mathchardef\ordinarycolon\mathcode`\:
  \mathcode`\:=\string"8000
  \begingroup \catcode`\:=\active
    \gdef:{\mathrel{\mathop\ordinarycolon}}
  \endgroup

\def\bigscpr(#1,#2){{\left(#1\nonscript \mskip2mu plus2mu \middle| \nonscript 
\mskip2mu
plus2mu#2\right)}}

\newcommand\indic{\operatorname{\bf 1}\nolimits}

\renewcommand\phi{\varphi}
\newcommand\eps{\varepsilon}
\renewcommand\epsilon{\varepsilon}

\newcommand{\N}{\mathbb{N}\nonscript\hskip.03em}

\newcommand{\K}{\mathbb{K}\nonscript\hskip.03em}
\newcommand\cA{\mathcal A}

\newcommand\cL{\mathcal L}
\newcommand\cK{\mathcal K}

\newcommand\cZ{\mathcal Z}
\newcommand\rd{\mathrm d}
\newcommand\Lpm{L_p(\mu)}
\newcommand\rr{\mathrm{r}}
\newcommand\cP{\mathcal P}

\newcommand\norm[1]{\|#1\|}

\newcommand\e{{\rm e}}

\def\formE(#1,#2){\sum_{e\in E}\int_{a_e}^{b_e} #1_e'(x)\ol{#2_e'(x)}\,dx}

\makeatletter
\let\qedhere@ams\qedhere
\def\qedhere{\@ifnextchar[{\@qedhere}{\qedhere@ams}}
\def\@qedhere[#1]{\tag*{\raisebox{-#1ex}{\qedhere@ams}}}
\makeatletter
\def\env@cases{%
  \let\@ifnextchar\new@ifnextchar
  \left\lbrace
  \def\arraystretch{1.1}%
  \array{@{\,}l@{\quad}l@{}}%
}
\renewcommand\section{\@startsection {section}{1}{\z@}%
                                     {-3.25ex \@plus -1ex \@minus -.2ex}%
                                     {1.5ex \@plus.2ex}%
                                     {\normalfont\large\bfseries}}
\makeatother

\newcommand\restrict{\vphantom f\mskip1mu\vrule\mskip2mu}
\newcommand\set[2]{\bigl\{#1{;}\penalty300\;#2\bigr\}}

\newcommand\ol{\overline}

\newcommand\comp{{\textnormal c}}
\newcommand\Cci{{\displaystyle 
C_{\raise0.2ex\hbox{$\scriptstyle\comp$}}^\infty}}

\renewcommand\le{\leqslant}

\renewcommand\ge{\geqslant}

\newcommand\sse{\subseteq}

\newcommand\di{\mathclose{}\,\mathrm{d}}

\newcommand\slim{\mathop{\rm s\kern.08em\mbox{\rm -}lim}} 

\newcommand\abstracttext{\noindent
We show that the formula for the essential norm of a multiplication operator on $L_p$, for $1<p<\infty$, also holds for $p=1$. We also provide a proof for the formula which works simultaneously for all $p\in[1,\infty)$.
\vspace{8pt}

\noindent
MSC 2010: 47B38, 46E30, 46B42.
\vspace{2pt}

\noindent
Keywords: Multiplication operator, $L_p$-space, compact operator, essential norm.
}
\begin{document}
\title{\mytitle}

\author{J\"urgen Voigt}

\date{}

\maketitle

\begin{abstract}
\abstracttext
\end{abstract}

\section{Introduction}
\label{sec-intro}

Let $(\Omega,\cA,\mu)$ be a $\sigma$-finite measure space, and let $1\le p<\infty$. For $u\in L_\infty(\mu)$ let $M_u$ be the bounded multiplication operator on $L_p(\mu)$ defined by
\[
 M_uf:=uf\qquad(f\in L_p(\mu)).
\]
Compactness properties of multiplication operators in various function spaces have been investigated in several papers; see \cite{Takagi-92}, \cite{Hudzik-K-K-06}, \cite{Bala-G-B-13}, \cite{Komal-P-R-16}, \cite{Raj-S-P-16}, \cite{Castillo-RF-SB-16}, \cite{Ramos-Fernandez-SB-17}, \cite{Castillo-R-RF-SB-19}, \cite{Ramos-Fernandez-RS-SB-19}. It is only in the recent paper \cite{Castillo-LA-RF21} that the
essential norm
\begin{equation}\label{eq-ess-norm}
\norm{M_u}_\e :=\inf\set{\norm{M_u+K}}{K\in\cK(L_p(\mu)},
\end{equation}
where $\cK(L_p(\mu))$ denotes the space of compact operators on $L_p(\mu)$, has been determined, for $1<p<\infty$. (The essential norm $\norm{M_u}_\e$ is the quotient norm in the Calkin algebra.) In order to describe this result we recall that the measure space can be decomposed as a disjoint union $\Omega=\Omega_d\cup\Omega_a$, where $\Omega_d,\Omega_a\in\cA$, the restriction $\mu_d$ of $\mu$ to $\Omega_d$ is a diffuse measure, and the restriction $\mu_a$ of $\mu$ to $\Omega_a$ is (purely) atomic. The property of being diffuse means that for every measurable subset $A$ of $\Omega_d$ with $\mu_d(A)>0$ there exists a measurable subset $A'$ of $A$ such that $0<\mu_d(A')<\mu_d(A)$. And the atomic part $\Omega_a$ is the union of a disjoint sequence $(B_n)_{n\in\N}$ of measurable sets, where each $B_n$ is an atom, which means that any measurable subset $B'$ of $B_n$ has measure $\mu_a(B')\in\{0,\mu_a(B_n)\}$. With this notation, the essential norm of $M_u$ is given by
\begin{equation}\label{eq-main}  
 \norm{M_u}_\e=\max\{\norm{u{\restrict}_{\Omega_d}}_\infty,\limsup_{n\to\infty}|u(B_n)|\}.
\end{equation}
(By $u(B_n)$ we denote the a.e.-value of $u$ on $B_n$; if $\mu_a(B_n)=0$ we choose $u(B_n):=0$.) The proof of \eqref{eq-main} given in \cite[Theorem~4.1]{Castillo-LA-RF21} does not carry over to the case~$p=1$. 

In Section~\ref{sec-ess-norm} we show that \eqref{eq-main} also holds for $p=1$. 
In Sections~\ref{sec-Omega-d-rev} and~\ref{sec-supplement} we provide a second -- quite different -- proof, which works simultaneously for all $p\in[1,\infty)$.

\section{The essential norm of $M_u$}
\label{sec-ess-norm}

Let $(\Omega,\cA,\mu)$ be a $\sigma$-finite measure space, and let $\Omega=\Omega_d\cup\Omega_a$ and $\Omega_a=\bigcup_{n\in\N}B_n$ be as described above.

\begin{theorem}\label{thm-ess-norm}
Let $u\in L_\infty(\mu)$, and let $M_u$ be the multiplication operator associated with $u$ on $L_1(\mu)$. Then $\norm{M_u}_\e$ is given by \eqref{eq-main}.
\end{theorem}

\begin{proof}
(i) For the inequality `$\le$' in \eqref{eq-main} we refer to \cite[first part of the proof of Theorem~4.1]{Castillo-LA-RF21}.

(ii) For the proof of `$\ge$' we first note that in the infimum of the formula \eqref{eq-ess-norm} (where in the present step we treat the general case $p\in[1,\infty)$)
one does not need all compact operators, but it is sufficient to consider operators leaving $L_p(\Omega_d)$ and $L_p(\Omega_a)$ invariant. Indeed, let $P_d$ and $P_a$ denote the canonical projections from $L_p(\mu)$
onto $L_p(\Omega_d,\mu_d)$ and $L_p(\Omega_a,\mu_a)$, respectively. Then for any bounded operator $S$ on $L_p(\mu)$ one has $\norm{(P_d-P_a)S(P_d-P_a)}\le\norm S$, and because of
\[
 P_dSP_d+P_aSP_a =\tfrac12\bigl((P_d+P_a)S(P_d+P_a) + (P_d-P_a)S(P_d-P_a)\bigr)
\]
one obtains $\norm{P_dSP_d+P_aSP_a} \le\norm S$. In view of $P_dM_uP_d+P_aM_uP_a=M_u$, this yields
\[
 \norm{M_u+P_dKP_d+P_aKP_a} \le \norm{M_u+K}
\]
for all compact operators, and $P_dKP_d+P_aKP_a$ is a compact operator leaving $L_p(\Omega_d)$ and $L_p(\Omega_a)$ invariant. As a consequence one also concludes that it is sufficient to prove the inequality `$\ge$' separately for diffuse and atomic measure spaces. For the remainder of the proof we now return to the case $p=1$.

(iii) In this part of the proof we show `$\ge$' for the case that $\Omega=\Omega_d$, i.e.~that $\mu$ is a diffuse measure. The case $u=0$ being trivial, assume that $\|u\|_\infty>0$ and let $0<\eps<\norm u_\infty$.
Then there exists a descending sequence $(A_n)_{n\in\N}$ in $\cA$ such that $0<\mu(A_n)\to0$ as $n\to\infty$ and $|u\restrict_{A_n}|\ge\|u\|_\infty-\eps$ for all $n\in\N$; without restriction $\mu(A_1)<\infty$. For $n\in\N$ put 
\[
 f_n:=\frac1{\mu(A_n)}\indic_{A_n},
\]
where $\indic_{A_n}$ denotes the indicator function of the set $A_n$. Let $K\in \cK(L_1(\mu))$. Because $(f_n)_{n\in\N}$ is a bounded sequence, the compactness of $K$ implies that there exists a subsequence $(f_{n_j})_{j\in\N}$ such that the sequence $(Kf_{n_j})$ is convergent; 
by passing to a subsequence, we can assume that $(Kf_n)$ is already convergent. Then there exists $n\in\N$ such that $\norm{Kf_n-Kf_m} \le \eps$ for all $m\ge n$. Choose $m\ge n$ large enough to obtain additionally $\frac1{\mu(A_m)}\ge 2\frac1{\mu(A_n)}$. Then one has
\[
 f_n-f_m=\tfrac1{\mu(A_n)}\indic_{A_n}-\tfrac1{\mu(A_m)}\indic_{A_m} = \tfrac1{\mu(A_n)}\indic_{A_n\setminus A_m}-\bigl(\tfrac1{\mu(A_m)}-\tfrac1{\mu(A_n)}\bigr)\indic_{A_m},
\]
$\bigl|f_n-f_m\bigr|\ge f_n$, $\norm{f_n-f_m}_1\ge \norm{f_n}_1=1$; hence
\begin{align*}
 \norm{(M_u+K)(f_n-f_m)}_1
 &\ge\norm{M_u(f_n-f_m)}_1-\norm{K(f_n-f_m)}_1\\
 &\ge(\norm u_\infty-\eps)\norm{f_n-f_m}_1 - \eps\\ 
 &\ge (\norm u_\infty - 2\eps)\norm{f_n-f_m}_1,
\end{align*}
$\norm{M_u+K}\ge\norm u_\infty - 2\eps$. As this holds for all $\eps\in(0,\norm u_\infty)$, we obtain $\norm{M_u+K}\ge\norm u_\infty$.

(iv) It remains to show that `$\ge$' holds in the case that $\Omega=\Omega_a$, i.e.~that $\mu$ is an atomic measure. If $\mu(B_n)\ne 0$ only for finitely many $n\in\N$, then $\limsup_{n\to\infty}|u(B_n)|=0$, and the assertion is trivial. Assume that this is not the case, without restriction $\mu(B_n)\ne0$ for all $n\in\N$. For $n\in\N$ let $P_n$ be the canonical projection from 
$L_1(\mu)$ onto $L_1(B_n)$, i.e.~$P_nf := \indic_{B_n}f$ ($f\in L_1(\mu)$),
and put $Q_n:=I-\sum_{j=1}^n P_j$. Iterating the procedure applied in step (ii) above one concludes that for all bounded operators $S$ on $L_1(\mu)$ and all $n\in\N$ one obtains $\norm{\sum_{j=1}^nP_jSP_j+Q_nSQ_n}\le\nobreak\norm S$. 

Given a compact operator $K\in\cK(L_1(\mu))$ we note that $\norm{Q_nK}\to0$ as $n\to\infty$. This holds because for any $g\in L_1(\mu)$ one has $\norm{Q_ng}\to 0$, and from the  relative compactness of $K(B_{L_1(\mu)}[0,1])$ (where $B_{L_1(\mu)}[0,1]$ denotes the closed unit ball of $L_1(\mu)$) together with the equicontinuity of the sequence $(Q_n)$ one concludes that
\[
\norm{Q_nK}= \sup_{\norm f\le1}\norm{Q_nKf}= \sup_{g\in K(B_{L_1(\mu)}[0,1])}\norm{Q_ng}\to 0\qquad (n\to\infty).
\]
In particular, we conclude that $\norm{P_nKP_n}\le\norm{(Q_{n-1}-Q_n)K}\to0$ ($n\to\infty$). Note that, for $n\in\N$, there exists $d_n\in\K$ such that $\norm{P_nKP_n}=|d_n|$ and $P_nKP_nf=d_n P_n f$ for all $f\in L_1(\mu)$. Hence the multiplication operator $D_K$, given by $L_1(\mu)\ni f\mapsto \sum_{n\in\N}d_nP_nf\in L_1(\mu)$, is a compact operator.

Now we estimate
\begin{align}\label{eq-comp-inequ-ell1}
\begin{split}
 \norm{M_u+K} &\ge\norm{\sum_{j=1}^nP_j(M_u+K)P_j + Q_n(M_u+K)Q_n}\\
 &=\norm{M_u+D_K+Q_n(K-D_K)Q_n}\\
 &\ge\norm{M_u+D_K}-\norm{Q_n(K-D_K)Q_n}.
\end{split}
\end{align}
From the argument given above we obtain $\norm{Q_n(K-D_K)Q_n}\le \norm{Q_n(K-D_K)} \to\nobreak0$ ($n\to\infty$), and from \eqref{eq-comp-inequ-ell1} we conclude that $\norm{M_u+K}\ge\norm{M_u+D_K}$. This shows that 
\begin{align*}
 \norm{M_u}_\e&=\inf\set{\norm{M_u+D}}{ D\text{ compact multiplication operator}}\\
 &= \inf\set{\sup_{n\in\N}|u(n)+d_n|}{(d_n)_{n\in\N}\text{ null sequence}}\\
 &=\limsup_{n\to\infty}|u(n)|. \qedhere
\end{align*}
\end{proof}

\begin{remark}
Step (iv) of our proof applies also to $p\in(1,\infty)$ and is an alternative to the last part of \cite[proof of Theorem 4.1]{Castillo-LA-RF21}. The idea of our proof is that $\norm{M_u+K}$ can be estimated from below by $\norm{M_u+D}$ for a suitable compact multiplication operator $D$. 
\end{remark}

\section{The case $\Omega=\Omega_d$, revisited}
\label{sec-Omega-d-rev}

Let $(\Omega,\cA,\mu)$ be a diffuse $\sigma$-finite measure space, $p\in[1,\infty)$, and let $u\in L_\infty(\mu)$. In this section we will present a proof of the equality
\begin{equation}\label{eq-ess-norm-d}
 \norm{M_u}_\e =\norm u_\infty
\end{equation}
(i.e.~\eqref{eq-ess-norm} for the present special case), which might throw a new light on this property.

We recall that an operator $S\in\cL(\Lpm)$ (the space of all bounded linear operators) is \textbf{positive}, $S\in\cL(\Lpm)_+$, if $Sf\ge0$ for all $f\in\Lpm_+$. Then $\cL^\rr(\Lpm)$, defined as the linear hull of $\cL(\Lpm)_+$, is the space of \textbf{regular operators}. It is a Banach lattice under the lattice operations
\begin{align*}
 (S\vee T)f&:=\sup\set{Sg+Th}{g,h\ge0,\ g+h=f},\\
 (S\land T)f&:=\inf\set{Sg+Th}{g,h\ge0,\ g+h=f} \qquad (f\in\Lpm_+)
 \end{align*}
 (valid for real operators $S,T\in \cL^\rr(\Lpm)$, the absolute value
\begin{equation*} 
 |S|f := \sup\set{|Sg|}{|g|\le f}\qquad (f\in\Lpm_+),
\end{equation*}
and with the \textbf{regular norm} $\norm S_\rr:=\norm{|S|}$. We refer to \cite[Chap.\,4]{Schaefer1974}, \cite[Sec.\,1.3]{Meyer-Nieberg1991} for more information.

As a preparation to the proof of \eqref{eq-ess-norm-d} we need the following property, where $q\in(1,\infty]$ denotes the exponent conjugate to $p$, $\frac1p+\frac1q=1$.

\begin{lemma}\label{lemma-decomp}
Let $\eta\in L_q(\mu)$ ($=L_p(\mu)'$), $g\in\Lpm$, $\eta, g\ge0$, $K\in\cL(\Lpm)$ defined by
\[
 Kf:=\Bigl(\int\eta f\di\mu\Bigr)\, g\qquad(f\in\Lpm).
\]
(Note that $K\in\cL(\Lpm)_+\sse\cL^\rr(\Lpm)$.) Let $u\in L_\infty(\mu)_+$ be such that $M_u\le K$.
Then $u=0$.
\end{lemma}

\begin{proof}
 Assume on the contrary that $u\ne0$. Then there exists $\eps>0$ such that $\mu([u\ge\eps])>0$ (with the notation $[u\ge\eps]:=\set{x\in\Omega}{u(x)\ge\eps}$). Further there exists $c>0$ such that
 $\mu([u\ge\eps]\cap[g\le c])>0$. Let $B\in\cA$, $B\sse[u\ge\eps]\cap[g\le c]$ with $0<\mu(B)<\infty$. Then $M_u\indic_B\ge\eps$ and $K\indic_B=\int_B\eta\di\mu\, g\le c\int_B\eta\di\mu$ on~$B$. There exists $B$ as above and such that $\int_B\eta\di\mu<\eps/ c$, and this leads to the contradiction $K\indic_B\le c\int\eta\di\mu<\eps\le M_u\indic_B$ on $B$. 
\end{proof}

The \textbf{centre} $\cZ(\Lpm)$ of $\cL(\Lpm)$ is the linear hull of the order interval
\[
 [-I,I] = \set{S\in \cL(\Lpm)}{-f\le Sf\le f\ (f\in\Lpm_+)}.
\]
Then $\cZ(\Lpm)\sse\cL^\rr(\Lpm)$ consists of the bounded multiplication operators and is isometrically isomorphic to $L_\infty(\mu)$; see \cite[C-I, Section~9]{Nagel1986}.

The centre $\cZ(\Lpm)$ is a projection band in the Banach lattice $\cL^\rr(\Lpm)$, i.e.~for all  $S\in\cL^\rr(\Lpm)$ there exists a (unique) decomposition $S=S_1+S_2$, where $S_1\in\cZ(\Lpm)$ and
\[
 S_2\in\cZ(\Lpm)^\rd = \set{T\in\cL^\rr(\Lpm)}{|T|\land R=0\ (R\in\cZ(\Lpm)_+)}; 
\]
see \cite[Chap.~II, Theorem~2.10]{Schaefer1974}
Let $\cP\colon\cL^\rr(\Lpm)\to\cZ(\Lpm)$, $S\mapsto S_1$ denote the associated band projection. What we have shown in Lemma~\ref{lemma-decomp} is that $\cP K=0$ for the special (positive) rank-one operators $K$. 
(Indeed, the lemma shows that $K$ belongs to $\cZ(\Lpm)^\rd$.) It is easy to see that any finite-rank operator $K\in\cL(\Lpm)$ can be written as a linear combination of rank-one operators as in Lemma~\ref{lemma-decomp}; hence $\cP K=0$ for all finite-rank operators.

\begin{theorem}\label{thm-main-rev}
Let $u\in L_\infty(\mu)$. Then 
\begin{equation}\label{eq-op-inequ-rev}
 \norm{M_u+K}\ge\norm{M_u} = \norm u_\infty\qquad(K\in\cK(\Lpm),
\end{equation}
and  \eqref{eq-ess-norm-d} holds.
\end{theorem}

\begin{proof}
Clearly, it suffices to show \eqref{eq-op-inequ-rev}. There are two ingredients of the proof:

(i) By the very definition, $\cP$ is contractive with respect to the regular norm (because band projections are contractive). However, it is shown in \cite[Theorem~1.4]{Voigt1988} that $\cP$ is also contractive with respect to the operator norm. This implies that $\cP$ can be extended by continuity to the closure of $\cL^\rr(\Lpm)$ in $\cL(\Lpm)$. In particular, for the extension one obtains $\cP K=0$ for all $K$ in the operator norm closure of the finite rank operators.

(ii) The space $\Lpm$ enjoys the \emph{approximation property}, i.e.~every compact operator on $\Lpm$ can be approximated in operator norm by finite rank operators. (We refer to \cite[Sections~3.4 and~4.1]{Megginson1998} for the approximation property.) This implies that $\cP K=0$ for all compact operators on $\Lpm$.

Putting together these two properties we obtain
\[
 \norm{M_u} =\norm{\cP(M_u+K)}\le \norm{M_u+K}
\]
for all $K\in\cK(\Lpm)$.
\end{proof}

\section{Supplement on the case $\Omega=\Omega_a$}
\label{sec-supplement}

We add that the case $\Omega=\Omega_a$ can be treated analogously to the case $\Omega=\Omega_d$ described in Section~\ref{sec-Omega-d-rev}. Then again the centre of $\cL(\Lpm)$ consists of the bounded multiplication operators. Lemma~\ref{lemma-decomp} is replaced by the property that multiplication operators are disjoint to positive rank-one operators~$K$ of the type
\begin{equation}\label{eq-rk-1-op}
 Kf = \int_{\Omega\setminus B_j}f\eta\di \mu\indic_{B_j}=\Bigl(\sum_{k\ne j}f(B_k)\eta(B_k)\mu(B_k)\Bigr)\indic_{B_j}\qquad (f\in\Lpm),
\end{equation}
where $j\in\N$ and $\eta\in L_q(\mu)_+$. Indeed, if $u\in L_\infty(\mu)_+$ is such that $M_u\le K$, then clearly $u(B_k)=0$ for all $k\ne j$. But 
$u(B_j)\indic_{B_j}=M_u\indic_{B_j}\le K\indic_{B_j}=0$; hence also $u(B_j)=0$. (Recall that $u(B_k)=0$ if $\mu(B_k)=0$, by our convention in the Introduction.)

The consequence is that, for a compact operator~$K$, its projection $\cP K$ onto the centre is the compact operator~$D_K$ (described in part (iv) of the proof of Theorem~\ref{thm-ess-norm}). This holds because for a compact operator~$K$ and $n\in\N$, the finite rank operator $(I-Q_n)K$ (with the notation of the proof of Theorem~\ref{thm-ess-norm}, part (iv)) can be decomposed as the multiplication operator $(I-Q_n)D_K$ and a linear combination of rank-one operators of the type \eqref{eq-rk-1-op}. As $(I-Q_n)K\to K$ ($n\to\infty$) in $\cL(\Lpm)$ and the band projection $\cP$ onto the centre is contractive with respect to the operator norm, one concludes that
$\cP K = \lim_{n\to\infty}\cP(I-Q_n)K = \lim_{n\to\infty}(I-Q_n)D_K=D_K$. 

Hence instead of \eqref{eq-comp-inequ-ell1} one obtains $\norm{M_u+K}\ge\norm{\cP(M_u+K)}=\norm{M_u+D_K}$, and the proof can be finished as in Section~\ref{sec-ess-norm}.

{
}
\bigskip

\noindent
J\"urgen Voigt\\
Technische Universit\"at Dresden\\
Fakult\"at Mathematik\\
01062 Dresden, Germany\\
{\tt 
juer\rlap{\textcolor{white}{xxxxx}}gen.vo\rlap{\textcolor{white}{yyyyyyyyyy}}%
igt@tu-dr\rlap{\textcolor{white}{%
zzzzzzzzz}}esden.de}

\end{document}